\documentclass[a4paper]{amsart}

\usepackage[utf8]{inputenc}
\usepackage{mathtools}
\usepackage{amssymb,amsfonts,amsmath}
\usepackage{enumitem}
\usepackage{color}
\usepackage{mathrsfs}
\usepackage{graphicx}
\usepackage{verbatim}
\usepackage{tikz-cd}
\usepackage{tikz}

\numberwithin{equation}{section}
\newtheorem{theorem}{Theorem}[section]
\newtheorem{corollary}[theorem]{Corollary}
\newtheorem{lemma}[theorem]{Lemma}
\newtheorem{proposition}[theorem]{Proposition}

\newtheorem{problem}[theorem]{Problem}

\newtheorem*{theorem*}{Theorem}

\theoremstyle{definition}
\newtheorem{definition}[theorem]{Definition}

\theoremstyle{remark}
\newtheorem{remark}[theorem]{Remark}

\newcommand{\F}{\mathbb{F}}
\newcommand{\N}{\mathbb{N}}
\newcommand{\Q}{\mathbb{Q}}

\newcommand{\Z}{\mathbb{Z}}

\newcommand{\prof}[1]{\mathbf{#1}}

\DeclareMathOperator{\LERF}{LERF}

\DeclareMathOperator{\ord}{ord}

\DeclareMathOperator{\PSL}{PSL}

\DeclareMathOperator{\T}{T}

\newcommand{\defeq}{\mathrel{\mathop{:}}=}

\renewcommand{\epsilon}{\varepsilon}

\title[Hereditarily just-infinite torsion groups]{Hereditarily just-infinite torsion groups with positive first $\ell^2$-Betti number}
\author[S. Kionke]{Steffen Kionke}
\author[E. Schesler]{Eduard Schesler}
\address{FernUniversit\"at in Hagen \\ Fakult\"at f\"ur Mathematik und Informatik \\
58084 Hagen}
\email{steffen.kionke@fernuni-hagen.de}
\email{eduard.schesler@fernuni-hagen.de}
\thanks{Funded by the Deutsche Forschungsgemeinschaft (DFG, German Research Foundation) - 441848266}

\begin{document}
\begin{abstract}
We present a new method to construct finitely generated, residually finite, infinite torsion groups. In contrast to known constructions, a profinite perspective enables us to control finite quotients and normal subgroups of these torsion groups. As an application, we describe the first examples of residually finite, hereditarily just-infinite groups with positive first $\ell^2$-Betti-number. In addition, we show that these groups have polynomial normal subgroup growth, which answers a question of Barnea and Schlage-Puchta.
\end{abstract}
\maketitle

\section{Introduction}
An infinite group $G$ is called \emph{just-infinite} if all of its proper quotients are finite.
Since its introduction by McCarthy~\cite{McCarthy68} in the late 1960's, the class of just-infinite groups remained an active field of research.
One reason for the importance of just-infinite groups is that, by Zorn's lemma, every finitely generated, infinite group admits a just-infinite quotient.

By a celebrated result of Wilson~\cite{Wilson71}, the study of just-infinite groups can be reduced to the study of  simple groups, branch groups and residually finite \emph{hereditarily just-infinite groups}, i.e.\ groups all of whose finite index subgroups are just-infinite.
Simple groups received a lot of attention and today a large diversity of simple groups with various properties is available; see e.g.~\cite{HydeLodha19,JuschenkoMonod13,OsinThom13,SkipperWitzelZaremsky19} for prominent examples.
Similarly, the class of  branch groups was intensively studied since Grigorchuk's discovery of the first example~\cite{Grigorchuk80} and a good understanding of numerous families of branch groups was developed; see~\cite{Grigorchuk00,BGS-branch} and references therein.
In contrast, the class of residually finite hereditarily just-infinite groups remained mysterious and, apart from some virtually cyclic groups, only two families of examples are known:
\begin{enumerate}
\item Centerless irreducible lattices in higher rank semisimple Lie groups -- for instance, $\PSL_n(\Z)$ for $n \geq 3$ -- and their $S$-arithmetic generalizations~\cite[Chapter IV]{Margulis91}.
\item The infinite $p$-torsion groups constructed by Ershov and Jaikin-Zapirain~\cite{ErshovJaikin13}. These groups are not linear and are locally zero-one, i.e.,
all of their finitely generated subgroups are either finite or of finite index.
\end{enumerate}
In the category of profinite groups more hereditarily just-infinite examples\footnote{A profinite group $G$ is understood to be (hereditarily) just-infinite, if each non-trivial closed normal subgroup of (a finite index subgroup of) $G$ has finite index in $G$.} are available; most prominently, the Nottingham groups~\cite{Johnson88} and some of their closed subgroups~\cite{Fesenko99}.
Interestingly, all of the just-mentioned examples of hereditarily just-infinite groups are virtually  residually-$p$ for some prime number $p$.
Wilson~\cite{Wilson10} asked whether this is always the case and
answered his question negatively in the profinite case (see also \cite{Vannacci16} for more examples).

Here we construct new finitely generated, hereditarily just-infinite torsion groups that answer Wilson's question negatively in the discrete case.
To formulate our main result, we need the notion of $\Pi$-graded groups. Let $\Pi = (p_i)_{i \in \N}$ be a sequence of prime numbers.
A group $G$ is called \emph{$\Pi$-graded} if it admits a residual chain, i.e.\ a chain of finite index normal subgroups $G = G_0 \trianglerighteq G_1 \trianglerighteq G_2 \ldots$ with $\bigcap_{i \in \N_0} G_i = \{1\}$, such that $G_{i-1} / G_i$ is an elementary abelian $p_i$-group. Every $\Pi$-graded group is residually finite.

\begin{theorem*}[Theorem~\ref{thm:main}]
Let $\Pi =(p_i)_{i\in \N}$ be a sequence of pairwise distinct primes, let $d \geq 2$ be a natural number, and let $\varepsilon > 0$.
There is a $d$-generated, hereditarily just-infinite, $\Pi$-graded torsion group $\Gamma$
with $b_1^{(2)}(\Gamma) > d-1-\varepsilon$.
\end{theorem*}

Here $b_1^{(2)}(\Gamma)$ denotes the first $\ell^2$-Betti number of $\Gamma$; see e.g.~\cite{Kammeyer19, Lueck:l2-invariants} for background on $\ell^2$-invariants.
Groups with positive first $\ell^2$-Betti number can be considered to be `big'.
On the other hand, just-infinite groups are arguably the smallest type of infinite groups.
It is therefore natural to ask how these notions of the size of a group are related to each other.
In the case of finitely generated simple groups, examples with positive first $\ell^2$-Betti numbers were constructed by Osin-Thom \cite{OsinThom13}.
Regarding the easy fact that branch groups always have vanishing first $\ell^2$-Betti numbers, the above theorem completes the picture of which types of finitely generated, just-infinite groups can admit a positive first $\ell^2$-Betti number.
Interestingly, the locally zero-one groups of Ershov and Jaikin-Zapirain lie on the other side of the spectrum as they tend to have property $(\T)$, which is a well-known obstruction to having a positive first $\ell^2$-Betti number (see~\cite{BekkaValette97}). 

The construction of finitely generated, residually finite torsion groups that are `big' in some sense received a lot of attention; examples were given  by Ershov~\cite{Ershov08}, Olshanskii and Osin~\cite{OlshanskiiOsin08}, and Osin~\cite{Osin09}.
The first examples with positive rank gradient were constructed independently by Osin~\cite{Osin11} and Schlage-Puchta~\cite{SchlagePuchta12}.
Shortly afterwards, L\"uck and Osin~\cite{LuckOsin11} were able to modify the construction in~\cite{Osin11} in order to obtain finitely generated, residually finite torsion groups with positive first $\ell^2$-Betti number.
All of the just-mentioned prior constructions of `big' torsion groups are based on taking direct limits of groups.
Starting with a group $G$ that is `big' in some sense -- e.g.\ by being hyperbolic or satisfying some inequality concerning a virtual or weighted notion of deficiency -- the idea is to successively quotient out high powers of elements of $G$, such that the corresponding quotient remains `big'.
Repeating this procedure for each element of $G$, while making sure that some appropriate properly descending residual chain of $G$ remains properly descending, guarantees that the limit of this process provides an infinite torsion group with the desired properties.

Our approach is inspired by these results, however, we develop a rather opposite strategy and build up our groups from finite to infinite.
More precisely, we inductively construct an inverse system
\[
G_0 \stackrel{q_1}{\longleftarrow} G_1 \stackrel{q_2}{\longleftarrow} G_2 \stackrel{q_3}{\longleftarrow} \cdots
\]
of finite groups $(G_i)_{i \in \N_0}$ endowed with generating sets $T_i = \{t_1^{(i)},\dots,t_d^{(i)}\}$ such that 
$G_{i} = V_{i} \rtimes G_{i-1}$ for a suitable $\F_{p_{i}}[G_{i-1}]$-module $V_{i}$ and $q_i(t_j^{(i)}) = t_j^{(i-1)}$ for all $i,j$.
Afterwards we will show that the inverse limit $G_{\infty} = \varprojlim_{i \to \infty} G_i$ is a hereditarily just-infinite profinite group and that the embedding
\[
\iota \colon \Gamma = \langle (t_1^{(i)})_{i \in \N},\ldots,(t_d^{(i)})_{i \in \N} \rangle \rightarrow G_{\infty}
\]
has a torsion image and extends to an isomorphism $\widehat{\iota} \colon \widehat{\Gamma} \rightarrow G_{\infty}$.
Using this, we will deduce that $\Gamma$ is hereditarily just-infinite as well.
This approach has the advantage that it neither involves a transition to the largest residually finite quotient of some group, nor does it involve a transition to an abstract just-infinite quotient. 
Since moreover $\widehat{\Gamma} \cong G_{\infty}$ and the groups $G_i$ can, at least in principle, be written down explicitly,
we have control on the finite quotients of $\Gamma$.
The cost of our approach is that the techniques for estimating the first $\ell^2$-Betti number that were used by L\"uck and Osin cannot be applied in our case.
However, in Proposition~\ref{prop:lowerboundL2} we overcome this issue by showing that the first $\ell^2$-Betti number of a finitely generated group $\Gamma$ is bounded below by $\limsup_{i \to \infty} \frac{ b_1(N_i,\F_{p_i})}{[\Gamma:N_i]}$, where $(N_i)_{i \in \N}$ is a residual chain of $\Gamma$ and $(p_i)_{i \in \N}$ is a  sequence of primes going to infinity.
This result, which might be of independent interest, has the advantage that it is not based on the specific construction of $\Gamma$.
As a consequence, it can be applied in a variety of situations.
For instance, it implies that already the $\Pi$-graded groups in Osin's earlier work~\cite{Osin11} have a positive first $\ell^2$-Betti number whenever $\Pi$ is a sequence of primes going to infinity.

It is natural to study interactions between various notions of `bigness' arising in group theory. 
For instance, Barnea and Schlage-Puchta \cite{BSP20} asked whether groups with positive rank gradient have large normal subgroup growth, i.e., of the fastest type $n^{\log(n)}$.
Here we answer this question negatively.
 In Proposition~\ref{prop:normal-growth} we show that topologically finitely generated $\Pi$-graded profinite groups (where $\Pi$ is a sequence of pairwise distinct primes) always have polynomial normal subgroup growth. 
The groups from our main theorem (and the ones constructed by Osin~\cite{Osin11}) have positive rank gradient, hence they provide counterexamples to \cite[Problem 5]{BSP20}.
 
\section{$\Pi$-graded groups and profinite completions}
In this paper, we deal with discrete groups and profinite groups. To distinguish these two cases, we denote profinite groups in bold face throughout.

The \emph{profinite completion} of a residually finite group $G$ is defined as the inverse limit $\prof{\widehat{G}} \defeq \varprojlim_{N \unlhd_f G} G / N$; here we write $N \unlhd_f G$ to indicate that $N$ is a normal subgroup of finite index in~$G$.
We recall that a profinite group $\mathbf{G}$ is said to be \emph{topologically finitely generated} if it contains a finitely generated, dense subgroup. 

\begin{definition}\label{def:CSP-chain-residual-chain}
Let $G$ be a residually finite group.
A descending chain $N_0 \geq N_1 \geq N_2 \geq \ldots$ of finite index normal subgroups $N_n \trianglelefteq_f G$ is called a \emph{residual chain}, if $\bigcap_{n \in \N} N_n = \{1\}$.
If moreover each finite index normal subgroup $N \leq G$ contains $N_n$ for some $n \in \N$, then we say that $(N_n)_{n \in \N}$ is \emph{cofinal}.
\end{definition}
When dealing with residually finite groups $G$, it often happens that there is a residual chain $(N_i)_{i \in \N}$ in $G$ that naturally arises and that one can understand.
In general, there is no reason to expect that this chain determines the profinite completion of $G$.
Only cofinal residual chains induce the full profinite topology on $G$.
More precisely, a residual chain $(N_n)_{n \in \N}$ of $G$ is cofinal  if and only if the canonical map
\[
\prof{\widehat{G}} \rightarrow \varprojlim \limits_{k \in \N} G / N_k,\ (g_N)_N \mapsto (g_{N_k})_{N_k}
\]
is an isomorphism.
It is well-known that every finitely generated residually finite group admits a cofinal residual chain.

Let $\Gamma$ be a group and let $p$ be a prime number. We write $D_p(\Gamma)$ for the kernel of the canonical map $\Gamma \to H_1(\Gamma,\F_p)$.
Let $\Pi = (p_i)_{i\in \N}$ be sequence of prime numbers. 
We define a normal series of $\Gamma$ recursively by $\Gamma^{\Pi(0)} := \Gamma$ and 
\[ \Gamma^{\Pi(n)} := D_{p_n}(\Gamma^{\Pi(n-1)}).\]

\begin{definition}
We say that $\Gamma$ is \emph{$\Pi$-graded}, if $(\Gamma^{\Pi(n)})_{n \in \N}$ is a residual chain.
\end{definition}
We note that every subgroup of a $\Pi$-graded group is $\Pi$-graded.

\begin{proposition}\label{prop:torsion}
Let $\Pi = (p_i)_{i \in \N}$ be a sequence of prime numbers going to infinity and let $\Gamma$ be a $\Pi$-graded torsion group. Then the following hold.
\begin{enumerate}
\item $(\Gamma^{\Pi(n)})_{n \in \N}$ is a cofinal residual chain.
\item Every factor of $\Gamma$ is $\Pi$-graded and, in particular, residually finite.
\end{enumerate}
\end{proposition}
\begin{proof}
Let $N \trianglelefteq \Gamma$ be a finite index normal subgroup of $\Gamma$. As the sequence $\Pi$ tends to infinity, we find some index $k \in \N$ such that $p_n$ does not divide $|\Gamma:N|$ for all $n \geq k$.
We note that the order of every element in $\Gamma^{\Pi(k)}$ is a product of primes $p_n$ with $n \geq k$. These orders are coprime to $|\Gamma:N|$ and therefore 
$\Gamma^{\Pi(k)} \subseteq N$.

Let $\pi\colon \Gamma \to \Delta$ be a surjective homomorphism. One can check by induction on $n$ that $\pi(\Gamma^{\Pi(n)}) = \Delta^{\Pi(n)}$ using that $\Gamma^{\Pi(n)}$ is a verbal subgroup of $\Gamma^{\Pi(n-1)}$. Let $g \in \Delta$ with $g \neq 1$ be given. Since $\Pi$ tends to infinity, we find $k \in \N$ such that the order of $g$ is not divisible by $p_n$ for all $n \geq k$. The order of any element from the $\pi$-preimage of $g$ is divisible by the order of $g$, hence $\Gamma^{\Pi(k)} \cap \pi^{-1}(g) = \emptyset$. This means, that $g$ is not contained in $\Delta^{\Pi(k)}$. We conclude that $\Delta$ is $\Pi$-graded.
\end{proof}
The second assertion implies that every normal subgroup of a $\Pi$-graded torsion group is closed in the profinite topology (if $\Pi$ tends to $\infty$). What about other subgroups? The following problem suggests itself:
\begin{problem}\label{problem:LERF}
Is every $\Pi$-graded torsion group $\LERF$ (if $\Pi$ tends to $\infty$)?
\end{problem}

Let $\prof{G}$ be a profinite group.
The series $\prof{G}^{\Pi(n)}$ is defined analogously using continuous homology.  If $\prof{G}$ is topologically finitely generated, then $(\prof{G}^{\Pi(n)})_{n\in \N}$ consists of open normal subgroups. We say that $\prof{G}$ is $\Pi$-graded if $\bigcap_{n\in\N} \prof{G}^{\Pi(n)} = \{1\}$.

\begin{corollary}\label{cor:profinite-completion-torsion-pi-graded}
Let $\Pi = (p_i)_{i \in \N}$ be a sequence of prime numbers going to infinity. 
If $\Gamma$ is a dense torsion subgroup of a $\Pi$-graded profinite group $\mathbf{G}$, then the canonical map $\prof{\widehat{\Gamma}} \to \mathbf{G}$ is an isomorphism.
\end{corollary}
\begin{proof}
As every residual chain of open normal subgroups in a profinite group is cofinal, we deduce that the profinite topology on $G$ induces the full profinite topology on $\Gamma$ and this implies the assertion. 
\end{proof}
\begin{remark}
Recall that a pair of residually finite groups $(\Gamma,\Delta)$ with $\Gamma \subsetneq \Delta$ is called a \emph{Grothendieck pair} if the inclusion $\iota \colon \Gamma \rightarrow \Delta$ induces an isomorphism $\widehat{\iota} \colon \prof{\widehat{\Gamma}} \rightarrow \prof{\widehat{\Delta}}$ of profinite completions.
Note that it follows from Corollary~\ref{cor:profinite-completion-torsion-pi-graded} that if $\Gamma_1,\Gamma_2$ are two dense torsion subgroups of a $\Pi$-graded profinite group $\mathbf{G}$ with
$\Gamma_1 \subsetneq \Gamma_2$, then $(\Gamma_1,\Gamma_2)$ is a Grothendieck pair.
An example of such a pair of finitely generated torsion groups, would give a negative answer to Problem~\ref{problem:LERF}.
\end{remark}

\begin{corollary}\label{cor:just-infinite}
Let $\Pi = (p_i)_{i \in \N}$ be a sequence of prime numbers going to infinity. 
Let $\prof{G}$ be a (hereditarily) just-infinite profinite $\Pi$-graded group. If $\Gamma \subseteq \prof{G}$ is a dense torsion subgroup, then $\Gamma$ is (hereditarily) just-infinite.
\end{corollary}
\begin{proof}
Since  $\prof{G}$ is the profinite completion of $\Gamma$, 
it is sufficient to prove the statement about just-infinite groups. We argue by contraposition and assume that $\Gamma$ is not just-infinite. Then $\Gamma$ has a proper infinite factor $\Delta$. As $\Gamma$ is a $\Pi$-graded torsion group, Proposition \ref{prop:torsion} implies that $\Delta$ is residually finite, hence its profinite completion $\prof{\widehat{\Delta}}$ is a proper infinite factor of $\prof{G} \cong \prof{\widehat{\Gamma}}$. In particular, $\prof{G}$ is not just-infinite.
\end{proof}

\begin{definition}\label{def:rank-gradient}
Let $\prof{G}$ be a topologically finitely generated profinite group. 
The \emph{rank gradient} of $\prof{G}$ is defined as 
\[\mathrm{RG}(\prof{G}) = \lim_{\prof{N}\trianglelefteq_o \prof{G}} \frac{d(\prof{N})}{[\prof{G} \colon \prof{N}]},\]
where  the limit is taken over the directed set of open normal subgroups and $d(\prof{N})$ denotes the minimal cardinality of a topological generating set of $\prof{N}$. We note that this net is monotonically decreasing by the Nielsen-Schreier formula and hence converges.
\end{definition}

\begin{lemma}\label{lem:just-infinite-rank-grandient}
Let $\prof{G}$ be a topologically finitely generated, just-infinite profinite group. If $\prof{G}$ has positive rank gradient, then $\prof{G}$ is hereditarily just-infinite.
\end{lemma}
\begin{proof}
We note that $\prof{G}$ does not contain an infinite abelian subnormal subgroup.  Indeed, by \cite[Cor.~3.8]{Wilson2000} this would imply that $\prof{G}$ is virtually abelian and thus has vanishing rank gradient. 

We argue by contraposition and assume that $\prof{G}$ is not hereditarily just-infinite. Then there is a closed subnormal subgroup $\prof{K} \leq_c \prof{G}$ of infinite index.  It follows from a result of Wilson \cite[Corollary 4.5]{Wilson2000} that $\prof{K}$ is a direct factor of an open normal subgroup $\prof{U} \trianglelefteq_o \prof{G}$, i.e. $\prof{U} = \prof{K} \times \prof{L}$ with an infinite closed subgroup $\prof{L}$.
This implies that the rank gradient of $\prof{G}$ vanishes. For completeness, we include the argument. Let $\prof{K}_n \trianglelefteq_o \prof{K}$ and $\prof{L}_n \trianglelefteq_o \prof{L}$ be residual chains of open normal subgroups. Then
\begin{align*}
	\mathrm{RG}(\prof{G}) &= \frac{1}{|\prof{G}:\prof{U}|} \mathrm{RG}(\prof{U}) \leq \lim_{n \to \infty} \frac{d(\prof{K}_n) + d(\prof{L}_n)}{|\prof{G}:\prof{U}|\cdot |\prof{K}:\prof{K}_n| \cdot |\prof{L}:\prof{L}_n|} \\
	&\leq \lim_{n \to \infty}\frac{d(\prof{K})}{|\prof{G}:\prof{U}| \cdot |\prof{L}:\prof{L}_n|} + \frac{d(\prof{L}) }{|\prof{G}:\prof{U}|\cdot |\prof{K}:\prof{K}_n|} = 0. \qedhere
\end{align*}

\end{proof}

\section{A result on the first $\ell^2$-Betti number}
Let $G$ be a finitely generated group. A \emph{finitely presented cover} of $G$ is a pair $(\Gamma,\pi)$ where $\Gamma$ is a finitely presented group and $\pi \colon \Gamma \to G$ is an epimorphism.
The following lemma strengthens a result of Ershov-L\"uck~\cite{ErshovLuck}.
\begin{lemma}\label{lem:finite-cover-l2}
 Let $G$ be a finitely generated group with a residual chain $(N_i)_{i\in\N}$. Then
 \[
 	b_1^{(2)}(G) = \inf_{(\Gamma,\pi)} \lim_{i \to \infty} \frac{b_1(\pi^{-1}(N_i),\Q)}{[G:N_i]}
 \]
where the infimum is taken over all finitely presented covers $(\Gamma,\pi)$ of $G$.
 \end{lemma}
 \begin{proof}
 The inequality $\leq$ follows directly from \cite[Theorem 3.1 (1)]{ErshovLuck}.
 
 For the converse inequality, we use the language of characters of groups and the approximation results from \cite{Kionke2018}. We pick a presentation $G \cong \langle S \mid R \rangle$  of $G$ with a finite generating set $S$.
 The set of relations $R$ may be infinite; nevertheless, we can write it as an increasing union $R = \bigcup_{j=1}^\infty R_j$ of finite subsets.
 Define $\Gamma_j = \langle S \mid R_j \rangle$ and let $\pi_j \colon \Gamma_j \to G$ denote the associated epimorphism.
 The presentation of $\Gamma_j$ provides us with a partial free resolution of $\Z$ over $\Z[\Gamma_j]$ that yields the chain complex
 \[
 	\mathcal{C}_j \colon \quad \Z[\Gamma_j]^{R_j} \stackrel{\partial^2_j}{\longrightarrow} \Z[\Gamma_j]^{S} \stackrel{\partial^1_j}{\longrightarrow} \Z[\Gamma_j]
 \]  
 of free $\Z[\Gamma_j]$-modules. Here $\partial^2_j$ is given by the Fox matrix. We denote by $\psi_j = \pi_j^*(\delta^{(2)})$ (resp.\ $\psi_{j,i}$) the pullback of the regular character of $G$ (resp.\ of $G/N_i$) to~$\Gamma_j$. For a character $\phi$ of $\Gamma_j$, we write $b_1^\phi(\mathcal{C}_j)$ to denote the first $\phi$-Betti number of~$\mathcal{C}_j$.
 In particular,
 \[
 	b_1^{\psi_{j,i}}(\mathcal{C}_j) = \frac{b_1(\pi_j^{-1}(N_i))}{[G:N_i]}.
 \]
 The characters $\psi_{j,i}$ are finite permutation characters and they converge pointwise to $\psi_j$, hence by \cite[Theorem 3.5]{Kionke2018}, we have
 \[
 	\lim_{i\to \infty}b_1^{\psi_{j,i}}(\mathcal{C}_j) = b_1^{\psi_j}(\mathcal{C}_j).
 \]
By definition, $b_1^{\psi_j}(\mathcal{C}_j)$ is the first Betti number of the complex
 \[
 	\ell^{2}(G)^{R_j} \stackrel{\partial^2_j}\longrightarrow \ell^2(G)^S \longrightarrow \ell^2(G).
 \]
 The closure of the union of the images of the $\partial_j^2$ is the closure of the image of $\partial^2$ in the complex
 	$\ell^{2}(G)^{R} \stackrel{\partial^2}\longrightarrow \ell^2(G)^S \longrightarrow \ell^2(G)$
 that computes the first $\ell^2$-Betti number of $G$. It follows from \cite[Theorem 1.12 (3)]{Lueck:l2-invariants} that $b_1^{(2)}(G) = \inf b_1^{\psi_j}(\mathcal{C}_j)$.
 \end{proof}
 \begin{proposition}\label{prop:lowerboundL2}
Let $G$ be a finitely generated group with residual chain $(N_i)_{i\in\N}$. Let $p_1,p_2,p_3,\dots$ be a sequence of prime numbers going to infinity.
Then the first $\ell^2$-Betti number satisfies
\[
	b_1^{(2)}(G) \geq  \limsup_{i \to \infty} \frac{ b_1(N_i,\F_{p_i})}{[G:N_i]}.
\]
\end{proposition}
\begin{proof}
Let $\pi\colon \Gamma \to G$ be a finitely presented cover of $G$. We write $M_i = \pi^{-1}(N_i)$. Then $b_1(N_i,\F_{p_i})\leq b_1(M_i,\F_{p_i})$.
Let $H_1(M_i,\Z) = \Z^{r_i} \oplus T_i$ where $T_i$ is a finite abelian group and $r_i = b_1(M_i,\Q)$ is the rank of the free part.
We write $d_i =  \dim_{\F_{p_i}} T_i \otimes_\Z \F_{p_i}$,
then $b_1(M_i,\F_{p_i}) = r_i + d_i$. 
Since $\Gamma$ is finitely presented, there is a constant $c> 1$ such that $p_i^{d_i} \leq |T_i| \leq c^{[\Gamma:M_i]}$ (this follows for instance from \cite[Lemma 6]{KKN17}). Since $[\Gamma:M_i]=[G:N_i]$ we
deduce
\[
	\frac{d_i}{[G:N_i]} \leq \frac{\log(c)}{\log{p_i}}
\]
and as $p_i$ tends to infinity with $i$, we have
\[
	\limsup_{i \to \infty} \frac{ b_1(N_i,\F_{p_i})}{[G:N_i]} \leq \lim_{i \to \infty} \frac{ b_1(M_i,\F_{p_i})}{[G:N_i]} = \lim_{i\to \infty} \frac{ r_i }{[G:N_i]} = \lim_{i\to \infty} \frac{ b_1(M_i,\Q)}{[G:N_i]}.
\]
Now the assertion follows from Lemma \ref{lem:finite-cover-l2}.
\end{proof}

\section{Semidirect products, normal subgroup growth and representations}
In this section, we discuss the main technical tools for our construction. To motivate the steps taken in this section, we record a fundamental observation.  Let $\prof{G}$ be a topologically finitely generated $\Pi$-graded profinite group and define $G_i = \prof{G}/\prof{G}^{\Pi(i)}$. If the primes in $\Pi$ are pairwise distinct, the theorem of Schur-Zassenhaus implies that $G_{i}$ is  a semidirect product of $G_{i-1}$ acting on an elementary abelian $p_{i}$-group. As a first step, we need to understand the involved semidirect products.

\subsection{Preliminaries on semidirect products}
We need the following elementary result on normal subgroups of semidirect products.
\begin{proposition}\label{prop:normal-subgroups}
Let $G$ be a finite group and let $p$ be a prime not dividing $|G|$. Let $V$ be an $\F_p[G]$-module and let $H = V \rtimes G$. The normal subgroups of $H$ are exactly the subgroups of the form
\[
	W \rtimes N
\]
where $N \trianglelefteq G$ is a normal subgroup of $G$ and $W \subseteq V$ is an $\F_p[G]$-submodule 
such that $N$ acts trivially on $V/W$.
\end{proposition}
\begin{proof}
Let $W$ and $N$ be as in the proposition. We check that $W \rtimes N$ is a normal subgroup. Let $n\in N$, $w\in W$, $g\in G$, $v \in V$ be given. Then 
\[
	(v,g)(w,n)(-g^{-1}v, g^{-1}) = (v + gw -gng^{-1}v, gng^{-1}).
\]
Since $N$ is normal in $G$, the conjugate $gng^{-1}$ lies in $N$.
Moreover, $N$ acts trivially on $V/W$ and we conclude that $gng^{-1}v - v \in W$, i.e. $v + gw -gng^{-1}v$ lies in $W$. This proves that $W \rtimes N$ is normal.

Conversely, let $M \trianglelefteq H$ be a normal subgroup. We define $W = M \cap V$; this is a normal subgroup of $H$ contained in $V$, i.e., an $\F_p[G]$-submodule. Let $q \colon H \to G$ denote the projection. We define $N = q(M)$;  this is a normal subgroup of $G$ as $q$ is onto.

We show that $N$ acts trivially on $V/W$. Let $n \in N$ and let $v \in V$. Then there is an element of the form $(u,n) \in M$. The commutator of $(-v,1)$ and $(u,n)$ lies in $M$ and is
 \[
         (-v,1)(u,n)(v,1)(-n^{-1}u,n^{-1}) = (-v + nv, 1);
\]
this means, $nv-v$ lies in $W$ and thus $n$ acts trivially on $V/W$.

Finally, we check that $M = W \rtimes N$. It is sufficient to show that $\{1\}\rtimes N$ lies in $M$.
Let $n \in N$ and pick $(u,n) \in M$. If $k$ denotes the order of $n$, then
\[
	(u,n)^k = (u+nu+ \dots+n^{k-1}u, 1)
\]
In other words, $u+nu+\dots+n^{k-1}u$ lies in $W$. As $n$ acts trivially on $V/W$, we deduce that $ku \in W$. By assumption $p \nmid k$ and hence $u \in W$. Consequently, $(1,n)$ lies in $M$.
\end{proof}
This allows us to deduce:
\begin{lemma}\label{lem:size}
Let $H = V \rtimes G$ where
$G$ is a finite group and let $V$ be an $\F_p[G]$-module with $p\nmid |G|$.
If $W \rtimes N \trianglelefteq H$ is a normal subgroup, then
\[
	\dim_{\F_p} W \geq \dim_{\F_p} V - \dim_{\F_p} V^N.
\]
\end{lemma}
\begin{proof}
The representation of $G$ on $V$ is semisimple. We pick a complementary submodule $U \subseteq V$. Then $U \cong V/W$ and hence $U \subseteq V^N$.
\end{proof}

\subsection{Parenthesis: Normal subgroup growth}
For a topologically finitely generated profinite group $\prof{G}$, we write $a_k^{\triangleleft}(\prof{G})$ to denote the number of open normal subgroups of index $k$ and write $s_k^{\triangleleft}(\prof{G}) = \sum_{\ell=1}^k a_k^{\triangleleft}(\prof{G})$. It is known that the sequence $s_k^{\triangleleft}(\prof{G})$ can be bounded from above by $k^{C \log(k)}$ for some constant $C$.
We say that $\prof{G}$ has \emph{polynomial normal subgroup growth}, if $s_k^{\triangleleft}(\prof{G})$ (or equivalently $a_k^{\triangleleft}(\prof{G})$) grows at most polynomially in~$k$.

In \cite[Problem 5]{BSP20} the question was raised, whether groups with positive rank gradient necessarily have normal subgroup growth of type $k^{\log(k)}$. Since there are $\Pi$-graded profinite groups with positive rank gradient -- e.g. the profinite completions of the groups constructed in Theorem \ref{thm:main} -- the next result provides a negative answer.
\begin{proposition}\label{prop:normal-growth}
Let $\Pi$ be a sequence of pairwise distinct primes.
Let $\prof{G}$ be a topologically finitely generated $\Pi$-graded profinite group. Then $\prof{G}$ has polynomial normal subgroup growth.
\end{proposition}
\begin{proof}
We define $G_i = \prof{G}/\prof{G}^{\Pi(i)}$ and we recall that by Schur-Zassenhaus $G_{i} \cong V_{i} \rtimes G_{i-1}$, where $V_{i}$ is an $\F_{p_{i}}[G_{i-1}]$-module. Assume that $\prof{G}$ is $d$-generated.
It is easy to check that $V_{i}$ is $d$-generated as an $\F_{p_{i}}[G_{i-1}]$-module.

Here is the crucial observation: $\prof{G}$ is prosolvable and thus has UBERG \cite[Cor.~6.12]{KionkeVannacci}.
This means that there is a constant $c \geq 1$ such that the number $R_n(\prof{G},\F_p)$ of irreducible representations of $\prof{G}$ over $\F_p$ of dimension at most $n$ is bounded above by $p^{cn}$ for all $n \in \N$ and all prime numbers $p$.

Put $b \defeq d+c+1$. We will show by induction on $i \in \N$ that $a_k^{\triangleleft}(G_i) \leq k^b$ for all $k\in \N$. For the base of the induction, we observe that $\F_{p_1}^d$ has at most $p_1^{rd}$ submodules of index $p_1^r$.

Now we assume that the assertion holds for $G_{i-1}$ and we consider the normal subgroups of $G_{i} = V_{i} \rtimes G_{i-1}$.
By Proposition~\ref{prop:normal-subgroups} every normal subgroup $M$ of $G_{i}$ has the form $W \rtimes N$ for a normal subgroup $N \trianglelefteq G_{i-1}$ and an $\F_{p_{i}}[G_{i-1}]$-submodule $W \subseteq V_i$. The index $k = |G_i:M|$ is $p_i^r\cdot|G_{i-1}:N|$ where $r$ is the codimension of $W$ in~$V_i$. We will show in the following that the number of submodules of index at most $p_i^r$ in $V_i$ is at most $p_i^{rb}$. This allows us to deduce
\[
	a^{\triangleleft}_{p_i^r\ell}(G_i) \leq p_i^{rb} a^{\triangleleft}_{\ell}(G_{i-1}) \leq (p_i^r \ell)^b
\]
from the induction hypothesis for all $r$ and all $\ell$ coprime to $p_i$.

The  non-degenerate pairing $V_i \times V_i^* \to \F_{p_i}$ between $V_i$ and its dual $V_i^*$ induces a bijection between the submodules of codimension $r$ in $V_i$ and the submodules of dimension $r$ in $V_i^*$. Recall that $V_i$ is $d$-generated as $\F_{p_i}[G_{i-1}]$-module and, as a consequence, the dual $V_i^*$ embeds into $\F_{p_i}[G_{i-1}]^d$.
This allows us to bound the number of submodules $W$ that have codimension at most $r$ in $V_i$ by the number $m_{r}$ of 
submodules of $\F_{p_i}[G_{i-1}]^d$ of dimension at most $r$. 

Let $\theta$ be an irreducible representation of $G_{i-1}$ over $\F_{p_i}$. Let $F := \mathrm{End}_{G_{i-1}}(\theta)$ be the endomorphism algebra. Then $F$ is a finite extension of $\F_{p_i}$; say $|F| = p_i^{e}$. It is well-known that the irreducible representation $\theta$ occurs with multiplicity $\mu(\theta) = \dim(\theta)/e$ in $\F_{p_i}[G_{i-1}]$.
Therefore we can bound the number of irreducible submodules of $\F_{p_i}[G_{i-1}]^d$ that are isomorphic to $\theta$ by
\[
|\mathrm{Hom}_{G_{i-1}}(\theta,\F_{p_i}[G_{i-1}]^d)| = |F|^{d\mu(\theta)} = p_i^{d\dim(\theta)}.
\]
In total, $\F_{p_i}[G_{i-1}]$ contains at most $I(j) \leq R_j(G_{i-1},\F_{p_i})p^{dj}\leq p_i^{(d+c)j}$ irreducible submodules of dimension $j$.

To conclude, we derive an upper bound for the number $m_r$ of all submodules of dimension at most $r$ in $\F_{p_i}[G_{i-1}]$.
For $a_1,\dots, a_r \in \N_0$ with $a_1 + 2a_2 + 3 a_3 +\dots + r a_r \leq r$ we can pick $a_j$ irreducibles of dimension $j$ to obtain a representation of dimension $a_1 + 2a_2+\dots+ra_r \leq r$; conversely, all representations of dimension at most $r$ arise in this way.
This allows us to conclude, using $\sum_{a_1+2a_2+\dots+ra_r = \ell} 1 \leq 2^{\ell-1}$ for $1\leq \ell \leq r$, that
\begin{align*}
	m_r &\leq \sum_{\ell=0}^r \sum_{a_1+2a_2+\dots+ra_r = \ell} \prod_{j=1}^r  I(j)^{a_j}\\
	&\leq \sum_{\ell=0}^r \sum_{a_1+2a_2+\dots+ra_r = \ell} \prod_{j=1}^r  (p_i^{(d+c)j})^{a_j}\\
	 &\leq p_i^{(d+c)r} \Big(\sum_{\ell=0}^r \sum_{a_1+2a_2+\dots+ra_r = \ell} 1\Big)\\
	 &\leq p_i^{(d+c)r} \Big(1 + \sum_{\ell=1}^r 2^{\ell-1}\Big)\\
	 &\leq p_i^{(d+c)r} 2^r \leq 
	 p_i^{(d+c+1)r} = p_i^{rb}.\qedhere
\end{align*}
\end{proof}

\subsection{The key step: Construction of $\F_p[G]$-modules}
Now we are able to describe the tool to construct the modules used in the iterated semidirect products.

Let $G$ be a finite group with a generating set $T = \{t_1,\dots,t_d\}$. Let $F$ denote a free group on $X=\{x_1,\dots,x_d\}$.
Let $\pi \colon F \to G$ be the presentation of $G$ defined by $x_i \mapsto t_i$. We write $R = \ker(\pi)$ for the set of relations.
Let $p$ be a prime number. We write $V_{T,p} = R/[R,R]R^p$ for the $\F_p[G]$ relation module.
It is well-known (from work of Gasch\"utz \cite{Gaschuetz54}) that if $p$ is coprime to $|G|$, then
\[
	V_{T,p} \cong \F_p[G]^{d-1} \oplus 1\!1_G
\]
where $1\!1_G$ denotes the trivial one-dimensional $\F_p[G]$-module.
In particular, for every subgroup $H \leq G$ we have
\begin{equation}\label{ineq:invariant-vectors}
	\dim_{\F_p} V_{T,p}^H = (d-1)|G:H|  + 1.
\end{equation}

\begin{theorem}\label{thm:module-existence}
Assume $p \nmid |G|$. Let $w_1,\dots,w_r \in F$ and let $H_1, \dots, H_s$ be subgroups of $G$.
Assume that 
\[
	\delta := 1 - \sum_{i=1}^r \frac{1}{\ord(\pi(w_i))(d-1)} - \sum_{j=1}^s \frac{|G|}{|H_j|\cdot |N_G(H_j)|} > 0.
\]
Then there is an $\F_p[G]$-module $V$ of dimension
\[
	\dim_{\F_p} V \geq (d-1) |G| \delta
\]
and elements $v_1, \dots, v_d \in V$ such that the  elements $\tilde{t}_i := (v_i,t_i) \in V \rtimes G$ for $i\in\{1,\dots,d\}$ generate $V \rtimes G$ and the following properties hold
\begin{enumerate}[label=(\alph*)]
\item\label{it:orders} $\ord(\tilde{\pi}(w_i)) = \ord(\pi(w_i))$ for all $i\in \{1,\dots,r\}$ where $\tilde{\pi}\colon F \to V \rtimes G$ is defined by $x_i \mapsto \tilde{t}_i$,
\item\label{it:invariants} $V^{H_j} = \{0\}$ for all $j\in \{1,\dots,s\}$,
\item\label{it:dimension} $\dim_{\F_p} V^K \leq \frac{1}{\delta |K|} \dim_{\F_p} V$ for all subgroups $K \leq G$.
\end{enumerate}
\end{theorem}
\begin{proof}
Recall that $V_{T,p} = R/[R,R]R^p$ denotes the $\F_p[G]$ relation module. We note that $F/[R,R]R^p \cong V_{T,p} \rtimes G$ since $p$ is coprime to $|G|$.
We define submodules $U_1,\dots,U_r$ of $V_{T,p}$ by
\[
	U_i := \langle w_i^{\ord(\pi(w_i))}[R,R]R^p \rangle_{\F_p[G]}.
\]
We observe that $w_i^{\ord(\pi(w_i))}[R,R]R^p$ is $\pi(w_i)$-invariant and hence
the cyclic module $U_i$ satisfies
\[
	\dim_{\F_p} U_i \leq |G/\langle \pi(w_i)\rangle| = \frac{|G|}{\ord(\pi(w_i))}.
\]
Moreover, we define the submodules $W_1,\dots,W_s$ by
\[
	W_j := \langle V_{T,p}^{H_j} \rangle_{\F_p[G]}
\]
and observe using \eqref{ineq:invariant-vectors} that $\dim_{\F_p }V_{T,p}^{H_j} =  (d-1)|G:H_j| +1$. As $V_{T,p}^{H_j}$ is stable under the action of the normalizer $N_G(H_j)$ and contains the non-trivial submodule $V_{T,p}^G$ we obtain 
\begin{align*}
\dim_{\F_p} W_j \leq |G:N_G(H_j)| (\dim_{\F_p} V_{T,p}^{H_j}-1) + 1 = \frac{(d-1)|G|^2}{|H_j|\cdot |N_G(H_j)|} + 1
\end{align*}
Moreover, each of the spaces $W_j$ contains the non-trivial submodule $V_{T,p}^G$, and so
\[
	\dim_{\F_p} \Bigl(\sum_{j=1}^s W_j \Bigr) \leq 1+ \sum_{j=1}^s \frac{(d-1)|G|^2}{|H_j|\cdot |N_G(H_j)|}
\]
Define $V = V_{T,p} / (\sum_i U_i + \sum_j W_j)$. Then
\begin{align*}
	\dim_{\F_p} V &\geq (d-1)|G| - \sum_{i=1}^r \frac{|G|}{\ord(\pi(w_i))} - \sum_{j=1}^s  \frac{(d-1)|G|^2}{|H_j|\cdot |N_G(H_j)|}\\
	&= (d-1)|G| \delta.
\end{align*}
If $r=0$ and $s=0$, we define $V = V_{T,p}/1\!1_G$, i.e., we quotient by a one-dimensional 
submodule with trivial $G$-action.
Then $V$ is a quotient of $V_{T,p}$ and hence $V \rtimes G$ is a quotient of $F/[R,R]R^p$. Let $\tilde{\pi}\colon F \to V \rtimes G$ be the induced map and we define $\tilde{t}_i := \tilde{\pi}(x_i)$. The properties \ref{it:orders} and \ref{it:invariants} are satisfied by construction. 

As a last step, we verify assertion \ref{it:dimension}. Let $K$ be any subgroup of $G$.
Using \eqref{ineq:invariant-vectors} and the observation that we quotient $V_{T,p}$ at least by a one-dimensional space of $G$-invariant vectors, we obtain
\[
	\dim_{\F_p} V^K \leq \dim_{\F_p} V_{T,p}^K -1 =  (d-1)|G:K|  \leq \frac{1}{\delta |K|}\dim_{\F_p} V. \qedhere
\]
\end{proof}

\section{Hereditarily just-infinite groups}

Let $d \geq 2$ be given and let $F$ be a free group of rank $d$ with basis $X = \{x_1,\ldots,x_d\}$.
Let $\Pi =(p_i)_{i\in \N}$ be a sequence of pairwise distinct prime numbers.
We construct an inverse system 
\begin{equation}\label{eq:inverse-limit}
G_0 \stackrel{q_1}{\longleftarrow} G_1 \stackrel{q_2}{\longleftarrow} G_2 \stackrel{q_3}{\longleftarrow} \cdots
\end{equation}
of finite $d$-generated groups $(G_k)_{k \in \N_0}$ together with presentations $\pi_k \colon F \to G_k$ such that $q_k \circ \pi_{k} = \pi_{k-1}$.
The group $G_{k+1}$ will be a semidirect product of the form $G_{k+1} = V_{k+1} \rtimes G_k$ for a suitable $\F_{p_{k+1}}[G_k]$-module $V_{k+1}$ that arises from Theorem~\ref{thm:module-existence}.
We start by fixing a positive real number $\varepsilon < \frac{1}{4}$ and a function $o \colon F \to \N$ with
\begin{equation}\label{eq:sum_o_w}
\sum_{w \in F} \frac{1}{o(w)} < \frac{\varepsilon}{2}.
\end{equation}
Let $G_0$ denote the trivial group, let $V_1 = \F_{p_1}^d$, let $G_1 = V_1 \rtimes G_0$, and let $q_1 \colon G_1 \rightarrow G_0$ be the trivial map.
We fix a generating set $T_1 = \{t_1^{(1)},\dots,t_d^{(1)}\}$ of $G_1$.
Suppose now that for some $k \in \N$ we have constructed a sequence
\begin{equation}\label{eq:inverse-system-till-k}
G_0 \stackrel{q_1}{\longleftarrow} G_1 \stackrel{q_2}{\longleftarrow} \cdots \stackrel{q_{k-1}}{\longleftarrow} G_{k-1} \stackrel{q_{k}}{\longleftarrow} G_k
\end{equation}
of finite groups $G_i$ with generating sets $T_i = \{t_1^{(i)},\dots,t_d^{(i)}\}$ such that the following properties hold for each $1 \leq i \leq k$.

\begin{enumerate}
\item The group $G_i$ splits as $G_i  = V_i \rtimes G_{i-1}$, where $V_i$ is an $\F_{p_{i}}[G_{i-1}]$-module.
\item The map $q_i$ is defined as the canonical projection $G_{i}  = V_{i} \rtimes G_{i-1} \rightarrow G_{i-1}$.
\item The presentation $\pi_i \colon F \to G_i$ given by $x_j \mapsto t_j^{(i)}$ for every $1 \leq j \leq d$ satisfies $q_i \circ \pi_{i} = \pi_{i-1}$.
\item If $\ord(\pi_{i-1}(w)) > o(w)$ for some $w \in F$, then $\ord(\pi_{i}(w)) = \ord(\pi_{i-1}(w))$.\label{it:stable-order}
\item\label{it:dimension-induction} The dimension of $V_i$ is bounded below by $(d-1) |G_{i-1}| (1- \varepsilon)$.
\item\label{it:invariants-induction} For each subgroup $K \leq G_{i-1}$ we have
\[
\dim_{\F_{p_i}} V_i^K \leq \frac{1}{(1-\varepsilon) |K|} \dim_{\F_{p_i}} V_i.
\]
\item\label{it:closure-induction} If $1 \leq i \leq k-1$ and
\begin{equation}\label{eq:logpi-condition}
\log(p_i) > \max\{\varepsilon^{-1}, \log(\varepsilon)(2\varepsilon- 2^{-1} )^{-1}\},
\end{equation}
then each normal subgroup $M$ of $G_{i+1}$ with non-trivial image in $G_{i-1}$ is of the form $V_{i+1} \rtimes N$, where $N$ is normal in $G_{i}$.
\end{enumerate}
Note that these assumptions are satisfied (and mostly vacuous) for $G_1 = V_1 \rtimes G_0$.
We list $w_1,\dots,w_r$  all words in $F$ such that $\ord(\pi_k(w_i)) > o(w_i)$.
If moreover
\begin{equation}\label{eq:logpk-condition}
\log(p_k) > \max\{\varepsilon^{-1},\log(\varepsilon)(2\varepsilon- 2^{-1})^{-1}\},
\end{equation}
then we make a list $H_1,\dots, H_s$ of the normal closures of all non-trivial elements of $G_{k-1}$ in $G_{k}$.
We note that by Proposition \ref{prop:normal-subgroups} each $H_s$ has the form $H_s = W \rtimes M$ with $|M| \geq 2$.
In this case it follows from Lemma~\ref{lem:size} and~\eqref{it:invariants-induction} that
 \[
\dim_{\F_{p_k}}W \geq \Big(1-\frac{1}{2 (1-\varepsilon)}\Big) \dim_{\F_{p_k}} V_k.
 \]
Together with~\eqref{it:dimension-induction}, this implies
\begin{equation}\label{eq:sum_H_i_part_1}
\sum_{j=1}^s \frac{1}{|H_j|}
\leq \frac{|G_{k-1}|}{2 |W|}
\leq \frac{|G_{k-1}|}{2 p_k^{(1-\frac{1}{2(1-\varepsilon)}) \dim_{\F_{p_k}} V_k}}
\leq \frac{|G_{k-1}|}{2 p_k^{(\frac{1}{2} - \varepsilon) (d-1) |G_{k-1}|}}.
\end{equation}
On the other hand, we can deduce from $\log(|G_{k-1}|) < |G_{k-1}|$ and condition~\eqref{eq:logpk-condition} that
$|G_{k-1}| < p_k^{\epsilon |G_{k-1}|}$.
By combining the latter inequality with~\eqref{eq:sum_H_i_part_1} we obtain

\begin{equation}\label{eq:sum_H_i_part_2}
\sum_{j=1}^s \frac{1}{|H_j|}
\leq \frac{p_k^{\epsilon |G_{k-1}|}}{2 p_k^{(\frac{1}{2}-\varepsilon) (d-1) |G_{k-1}|}}
= \frac{1}{2} p_k^{ ( \epsilon - (\frac{1}{2}-\varepsilon) (d-1) ) |G_{k-1}|}.
\end{equation}
Next we observe that the assumption $\varepsilon \in (0,\frac{1}{4})$ implies $\epsilon < 2^{-1} -\varepsilon$.
In view of this, we see that
\begin{equation}\label{eq:epsilon-1-over-2}
(\epsilon - (2^{-1} -\varepsilon) (d-1) ) |G_{k-1}|
\leq \epsilon - (2^{-1} -\varepsilon) (d-1)
\leq 2\epsilon - 2^{-1}.
\end{equation}
Now we can insert~\eqref{eq:epsilon-1-over-2} into~\eqref{eq:sum_H_i_part_2} and apply~\eqref{eq:logpk-condition} to deduce that

\begin{equation}\label{eq:final-1}
\sum_{j=1}^s \frac{1}{|H_j|}
\leq \frac{1}{2} p_k^{2\epsilon - 2^{-1}}
< \frac{\varepsilon}{2}.
\end{equation}
Finally, we can combine~\eqref{eq:sum_o_w} and~\eqref{eq:final-1} to obtain
\begin{align*}
\delta &:= 1 - \sum_{i=1}^r \frac{1}{\ord(\pi_k(w_i))(d-1)} - \sum_{j=1}^s \frac{|G_k|}{|H_j|\cdot |N_{G_k}(H_j)|}\\
& \geq 1 - \sum_{w \in F} \frac{1}{o(w)} - \sum_{j=1}^s \frac{1}{|H_j|}\\
& > 1-\varepsilon.
\end{align*}
In this case Theorem \ref{thm:module-existence} provides us with an $\F_{p_{k+1}}[G_k]$-module $V_{k+1}$ of dimension
\[
\dim_{\F_{p_{k+1}}} V_{k+1}
\geq (d-1) |G_k| \delta
> (d-1) |G_k| (1 - \varepsilon)
\]
and elements $t_i^{(k+1)} = (v_i,t_i^{(k)}) \in G_{k+1} \defeq V_{k+1} \rtimes G_k$ for $i\in\{1,\dots,d\}$ such that $T_{k+1} \defeq \{t_1^{(k+1)},\dots,t_d^{(k+1)}\}$ generates $G_{k+1}$ and the following properties hold.
\begin{enumerate}[label=(\alph*)]
\item\label{it:orders-end} $\ord(\pi_{k+1}(w_i)) = \ord(\pi_k(w_i))$ for all $i\in \{1,\dots,r\}$, where $\pi_{k+1}\colon F \to G_{k+1}$ is defined by $x_i \mapsto t_i^{(k+1)}$,
\item\label{it:invariants-end} $V_{k+1}^{H_j} = \{0\}$ for all $j\in \{1,\dots,s\}$,
\item\label{it:dimension-end} $\dim_{\F_{p_{k+1}}} V_{k+1}^K \leq \frac{1}{(1-\varepsilon)|K|} \dim_{\F_{p_{k+1}}} V_{k+1}$ for all subgroups $K \leq G_{k}$.
\end{enumerate}
Altogether we have extended the sequence~\eqref{eq:inverse-system-till-k} to a sequence
\[
G_0 \stackrel{q_1}{\longleftarrow} G_1 \stackrel{q_2}{\longleftarrow} \cdots \stackrel{q_{k}}{\longleftarrow} G_{k} \stackrel{q_{k+1}}{\longleftarrow} G_{k+1}
\]
that satisfies the conditions (1)-(6) for $k+1$ instead of $k$.
To verify~\eqref{it:closure-induction}, we assume
\[
\log(p_k) > \max\{\varepsilon^{-1},\log(\varepsilon)(2\varepsilon- 2^{-1})^{-1}\}.
\]
Let $M$ be a normal subgroup of $G_{k+1}$ with non-trivial image $M_{k-1}$ in $G_{k-1}$.
Then we know from Proposition~\ref{prop:normal-subgroups} that the image $M_k$ of $M$ in $G_k$ contains the normal closure of some non-trivial element of $G_{k-1}$, i.e., it contains some group $H_j$ used in the construction.
In view of~\ref{it:invariants-end} we therefore have $V_{k+1}^{M_k} = \{0\}$.
In this case, it follows from another application of Proposition~\ref{prop:normal-subgroups} and Lemma~\ref{lem:size} that $M = V_{k+1} \rtimes M_k$, which completes the induction.

\smallskip

\noindent We can therefore consider the inverse limit $\prof{G}_{\infty} = \varprojlim_{k \to \infty} G_k$ of the sequence~\eqref{eq:inverse-limit}.
Let $\pi \colon F \to \prof{G}_\infty$ denote the homomorphism that is induced by the projections $\pi_i \colon F \rightarrow G_i$ and let $\Gamma = \pi(F)$.

\begin{theorem}\label{thm:main}
The group $\Gamma$ is a residually finite, $d$-generated, hereditarily just-infinite torsion group with 
\[
b_1^{(2)}(\Gamma) \geq (d-1) (1-\varepsilon).
\]
Moreover, $\prof{G}_{\infty}$ is hereditarily just-infinite as well and the inclusion $\iota \colon \Gamma \rightarrow \prof{G}_{\infty}$ extends to an isomorphism $\widehat{\iota} \colon \widehat{\prof{\Gamma}} \rightarrow \prof{G}_{\infty}$.
\end{theorem}
\begin{proof}
The group $\Gamma$ is a subgroup of a profinite (hence residually finite) group and is therefore residually finite. By construction $\Gamma$ is dense in $\prof{G}_\infty$.
Next we show that $\Gamma$ is a torsion group.
To this end, let $w \in F$ be arbitrary.
If $\ord(\pi_k(w)) > o(w)$ for some $k$, then it follows from~\eqref{it:stable-order} that $\ord(\pi_\ell(w)) = \ord(\pi_{k}(w))$ for all $\ell \geq k$ and hence $\ord(\pi(w)) = \ord(\pi_k(w))$.
On the other hand, if $\ord(\pi_k(w))$ is bounded above by $o(w)$ for every $k$, then the order of $\pi(w)$ is bounded above by $o(w)$ and is finite.
From Corollary~\ref{cor:profinite-completion-torsion-pi-graded} we can therefore deduce that $\widehat{\iota}$ is an isomorphism.
To verify the lower bound on $b_1^{(2)}(\Gamma)$, we observe that the kernel $N_k$ of the projection $\Gamma \to G_k$ maps onto $V_{k+1}$.
As a consequence, it follows from~\eqref{it:dimension-induction} that
\[
b_1(N_k,\F_{p_{k+1}})
\geq \dim_{\F_{p_{k+1}}} V_{k+1}
\geq (d-1)(1-\varepsilon)|G_k|.
\]
In other words, we have
\[
\frac{b_1(N_k,\F_{p_{k+1}})}{[G \colon N_k]}
\geq (d-1)(1-\varepsilon).
\]
In this case Proposition~\ref{prop:lowerboundL2} tells us that $b_1^{(2)}(\Gamma) \geq (d-1) (1-\varepsilon)$.
It remains to show that $\Gamma$ is hereditarily just-infinite. By Corollary \ref{cor:just-infinite} it is sufficient to show that $\prof{G}_\infty$ is hereditarily just-infinite.
Since $\prof{G}_\infty$ has rank gradient larger than $(1-\varepsilon)(d-1)$, it is sufficient to show that $\prof{G}_\infty$ is just-infinite by Lemma~\ref{lem:just-infinite-rank-grandient}.
Let therefore $M$ be a non-trivial closed normal subgroup of $\prof{G}_\infty$.
We need to show that $M$ is open.
Let $M_i := \pi_i(M)$ for each $i$.
We observe that~\eqref{eq:logpi-condition} is satisfied for each sufficiently large number $i$.
In particular, if $i$ is such a large number and $M_{i}$ is non-trivial, then~\eqref{it:closure-induction} tells us that $M_{i+2} = V_{i+2} \rtimes M_{i+1}$.
Since this holds for all sufficiently large $i$, it follows that $M$ is open.
\end{proof}

\begin{remark}\label{rem:flexible-finite-image}
(a) The construction of $\Gamma$ and $\textbf{G}_{\infty}$ can be generalized by replacing the trivial group $G_0$ with an arbitrary finite group.
To this end, one only has to add the condition that the primes in the sequence $\Pi =(p_i)_{i\in \N}$ do not divide the order of $G_0$.
In this case, the heredirarily just-infinite group $\Gamma$ in Theorem~\ref{thm:main} has the additional property that it maps onto $G_0$.

(b) If one is only interested in applying our construction in order to produce new finitely generated, infinite torsion groups, then the construction can be significantly shortened.
Indeed, in this case the only steps that remain necessary are Proposition~\ref{prop:normal-subgroups}, part $(a)$ of Theorem~\ref{thm:module-existence}, and a simplified variant of the induction in the current section.
\end{remark}

\bibliographystyle{amsplain}
\bibliography{literatur}

\end{document}